\documentclass[a4paper,12pt,reqno]{amsart}

\usepackage[all,cmtip]{xy}
\usepackage{amsmath}
\usepackage{amssymb}
\usepackage{eucal}
\usepackage{mathtools}
\usepackage{color}
\definecolor{cobalt}{RGB}{61,89,171}
\usepackage[colorlinks,citecolor=cobalt,linkcolor=cobalt,urlcolor=cobalt,pdfpagemode=UseNone,backref = page]{hyperref}

\newcommand*{\defeq}{\mathrel{\vcenter{\baselineskip0.5ex \lineskiplimit0pt
                     \hbox{\scriptsize.}\hbox{\scriptsize.}}}%
                     =}

\newcommand{\id}{\textrm{id}}

\newtheorem{theorem}{Theorem}[section]
\newtheorem*{theorem*}{Theorem}

\newtheorem{corollary}[theorem]{Corollary}

\theoremstyle{definition}

\theoremstyle{remark}
\newtheorem{remark}[theorem]{Remark}

\numberwithin{equation}{section}

\newcommand{\bN}{\mathbb{N}}
\newcommand{\bQ}{\mathbb{Q}}
\newcommand{\bR}{\mathbb{R}}
\newcommand{\bZ}{\mathbb{Z}}

\newcommand{\fg}{\mathfrak{g}}
\newcommand{\fh}{\mathfrak{h}}

\newcommand{\bs}{\backslash}

\newcommand{\Aut}{\mathrm{Aut}}
\newcommand{\Aff}{\mathrm{Af{}f}}

\begin{document}

\title{Vaisman nilmanifolds}


\author{Giovanni Bazzoni}
\address{Philipps-Universit\"at Marburg\\FB Mathematik \& Informatik\\
Hans-Meerwein-Str.\ 6 (Campus Lahnberge), 35032, Marburg}
\curraddr{}
\email{bazzoni@mathematik.uni-marburg.de}
\thanks{}



\date{}

\begin{abstract}
We prove that if a compact nilmanifold $\Gamma\bs G$ is endowed with a Vaisman structure, then $G$ is isomorphic to the Cartesian product of the Heisenberg group with $\bR$. 
\end{abstract}

\maketitle

\section{Introduction and preliminaries}\label{Sec:1}

In this section we recall a few facts about nilmanifolds and about the geometric structures that play a role in this note; we take advantage of this preamble to put our result 
in the right perspective.

A \emph{nilmanifold} $N$ is a compact quotient of a connected, simply connected, nilpotent Lie group $G$ by a lattice $\Gamma$. Every such $G$ is dif{}feomorphic to $\bR^n$ for some $n$, hence the 
natural projection $G\to N$ is the universal cover and $\pi_1(N)=\Gamma$. In particular, nilmanifolds are aspherical spaces. Being a subgroup of a nilpotent group, $\Gamma$ is nilpotent.

Recall that a space $X$ is nilpotent if $\pi_1(X)$ is a nilpotent group and $\pi_n(X)$ is a nilpotent $\pi_1(X)$-module for $n\geq 2$. Since nilmanifolds are aspherical spaces with nilpotent 
fundamental groups, they are nilpotent spaces. Therefore, the fundamental result of rational homotopy theory implies that the rational homotopy type of a 
nilmanifold is codified in its minimal model (see \cite{Halperin}). By a result of Hasegawa (\cite{Hasegawa}), the minimal model of a nilmanifold $N=\Gamma\bs G$ is isomorphic to the 
Chevalley-Eilenberg complex $(\Lambda\,\fg^*,d)$, where $\fg$ is the Lie algebra of $G$ and $\fg^*$ is its dual (see also \cite[Theorem 3.18]{FOT}). In particular, the minimal model of a 
nilmanifold is generated in degree 1. Notice that, by a result of Mal'cev, a connected and simply connected nilpotent Lie group admits a lattice if and only if its Lie algebra 
admits a basis with respect to which the structure constants are rational numbers. Thus the minimal model of a nilmanifold $N=\Gamma\bs G$ is automatically defined over $\bQ$.

Simple examples of nilmanifolds are provided by tori; more complicated examples are the so-called Heisenberg nilmanifolds. These are quotients of the Heisenberg group $H(1,n)$ ($n\geq 1$) 
by a lattice. Recall that $H(1,n)$ can be identified with the matrix group

\[
 H(1,n)=\left\{\begin{pmatrix} 1 & y_1 & y_2 & \ldots & y_n & z\cr
0 & 1 & 0 & \ldots & 0 & x_1\cr
\vdots & 0 & \ddots & \ddots & \vdots & x_2\cr
\vdots & \vdots & \ddots & \ddots & 0 & \vdots\cr
\vdots & \vdots &  & \ddots & 1 & x_n\cr
0 & 0 & \ldots & \ldots & 0 & 1\cr
\end{pmatrix} \ | \ x_i,y_i,z\in\bR, \ i=1,\ldots,n
\right\}
\]

By considering a nilmanifold quotient of $H(1,1)\times\bR$, Thurston gave in \cite{Thurston} the first example of a compact symplectic manifold without any K\"ahler structure; indeed, any compact 
quotient of $H(1,1)\times\bR$ has $b_1=3$. Every K\"ahler manifold is formal \cite{DGMS}, 
but Hasegawa proved (\cite{Hasegawa}) that a nilmanifold is formal if and only if it is dif{}feomorphic to a 
torus. It turns out that a nilmanifold quotient of $H(1,n)\times \bR$ admits a symplectic structure if and only if $n=1$. 

Although tori are the only K\"ahler nilmanifolds, a nilmanifold quotient of $H(1,n)\times \bR$ ($n\geq 1$) admits a special kind of Hermitian structure, namely a locally conformal K\"ahler structure. A Hermitian structure $(g,J)$ on a 
$2n$-dimensional manifold $M$ ($n\geq 2$) is called \emph{locally conformal K\"ahler}, \emph{lcK} for short, if the fundamental form $\omega\in\Omega^2(M)$, defined by $\omega(X,Y)=g(JX,Y)$, 
satisfies the equation
\[
 d\omega=\theta\wedge\omega,
\]
where $\theta\in\Omega^1(M)$ is the Lee form, defined by $\theta(X)=-\frac{1}{n-1}(\delta\omega)(JX)$. Clearly, if $\theta=0$, the Hermitian structure is K\"ahler. If $\theta$ is exact, 
then one can make a \emph{global} conformal change of the metric so that the corresponding rescaled fundamental 2-form is closed. In general, $\theta$ is only locally exact; hence one can always 
find a \emph{local} conformal change of the metric such that the corresponding rescaled fundamental 2-form is closed. When talking about 
lcK structures, it is customary to assume that $\theta$ is nowhere vanishing. If the manifold is compact, this prevents $\theta$ from being exact. LcK metrics on compact complex surfaces have been 
studied by Belgun, \cite{Belgun}; homogeneous lcK structures have been investigated in \cite{ACHK,GMO,HK}. Recently, lcK manifolds have attracted interest in Physics, 
as possible ``dif{}ferent'' compactifications of manifolds, arising in supersymmetry constructions, which locally posses a K\"ahler structure (see \cite{Shahbazi1,Shahbazi2}). We refer the reader 
to \cite{DO} for further details on lcK geometry.

Concerning lcK structures on nilmanifolds, it was conjectured by Ugarte in \cite{Ugarte} that if a nilmanifold $\Gamma\bs G$ of dimension $2n+2$ is endowed with a lcK structure, then $G$ is 
isomorphic to the product of $H(1,n)$ with $\bR$. This conjecture was proven by Sawai in \cite{Sawai} under the additional assumption that the complex structure of the lcK structure is 
left-invariant. As a byproduct of \cite{BazzoniMarrero} (see also \cite{BazzoniMarrero2}), this conjecture was proven for 4-dimensional nilmanifolds. However, it remains open for lcK structures 
on nilmanifolds of dimension $\geq 6$ with non-left-invariant complex structure.

\emph{Vaisman structures} are a particular type of lcK structures, those for which the Lee form is parallel with respect to the Levi-Civita connection. They were introduced by Vaisman in 
\cite{Vaisman1,Vaisman} under the name of generalized Hopf structures. Akin to the K\"ahler case, it turns out that when the underlying 
manifold is compact, the existence of a Vaisman structure influences the topology; for instance, the first Betti number of a compact Vaisman manifold is odd. However, there exist compact lcK manifolds whose first Betti number is even (see \cite{OT}). Recently, a Hard Lefschetz property for compact Vaisman manifolds has been studied in \cite{CMDNMY2}. As 
an application, the authours constructed locally conformal symplectic manifolds of the first kind with no compatible Vaisman metrics (see also \cite{BazzoniMarrero} for more such examples).

The main result of this paper, proved in Section \ref{sec:2}, is a positive answer to the conjecture of Ugarte in case the lcK structure is Vaisman:

\begin{theorem*}
 Let $V=\Gamma\bs G$ be a compact nilmanifold of dimension $2n+2$, $n\geq 1$, endowed with a Vaisman structure. Then $G$ is isomorphic to $H(1,n)\times\bR$.
\end{theorem*}

Vaisman structures are closely related to Sasakian structures. Recall that a \emph{Sasakian structure} on a manifold $S$ of dimension $2n+1$ consists of a Riemannian metric $h$, a tensor $\Phi$ of 
type $(1,1)$, a 1-form $\eta$ and a vector field $\xi$ subordinate to the following relations:
\begin{itemize}
 \item $\eta(\xi)=1$;
 \item $\Phi^2=-\id+\eta\otimes\xi$;
 \item $d\eta(\Phi X,\Phi Y)=d\eta(X,Y)$;
 \item $h(X,Y)=d\eta(\Phi X,Y)+\eta(X)\eta(Y)$;
 \item $N_\Phi+2d\eta\otimes \xi=0$,
\end{itemize}
where $N_\Phi$ is the Nijenhuis torsion of $\Phi$ - see  \cite{BoyerGalicki}. If $S$ is a Sasakian manifold and $\varphi\colon S\to S$ is a Sasakian automorphism, i.e.\ a dif{}feomorphism which 
respects the whole Sasakian structure, one can show that the mapping torus $S_\varphi$, defined as the quotient space of $S\times \bR$ by the $\bZ$-action generated by
\begin{equation}\label{mapping:torus}
1\cdot(s,t)=(\varphi(s),t+1), \quad (s,t)\in S\times\bR
\end{equation}
admits a Vaisman structure; clearly, taking the identity as Sasakian automorphism of $S$, we see that the product $S\times S^1$ admits a Vaisman structure. 
A converse of this result holds: if $V$ is a compact manifold endowed with a Vaisman structure, then one finds a Sasakian manifold $S$ and a Sasakian automorphism 
$\varphi\colon S\to S$ such that $V$ is dif{}feomorphic to $S_\varphi$ (see \cite[Structure Theorem]{OV} and \cite[Corollary 3.5]{OV2}).

The interplay between Vaisman and Sasakian structures on compact manifolds plays a crucial role in this paper.

A left-invariant Vaisman structure on $H(1,n)\times\bR$, descending to a Vaisman structure on every compact quotient of $H(1,n)\times\bR$ by a lattice, was first described in \cite{CFdL}. We recall it briefly.

The Lie algebra $\fh(1,n)$ of $H(1,n)$ has a basis $\{X_1,Y_1,\ldots,X_n,Y_n,\xi\}$ whose only non-zero Lie brackets are
\[
 [X_i,Y_i]=-\xi \quad \forall \ i=1,\ldots,n
\]
We take $h$ to be the metric making $\{X_1,Y_1,\ldots,X_n,Y_n,\xi\}$ orthonormal and define $\Phi$ by setting $\Phi(X_i)=Y_i$, $\Phi(Y_i)=-X_i$, $i=1,\ldots,n$, and $\Phi(\xi)=0$.
If $\{\alpha_1,\beta_1,\ldots,\alpha_n,\beta_n,\eta\}$ is the dual basis of $\fh(1,n)^*$, $(h,\Phi,\eta,\xi)$ is a Sasakian structure on $\fh(1,n)$. On the sum $\fh(1,n)\oplus\bR$, a Vaisman structure is given by taking the metric $g=h+dt^2$, where $\partial_t$ generates the $\bR$-factor, and the integrable almost complex structure $J$ defined by
\[
 J(X_i)=Y_i, \quad J(Y_i)=-X_i \quad \text{and} \quad J(\xi)=\partial_t, \ i=1,\ldots,n.
\]
Here the Lee form is $\theta=dt$. Hence $(h,\Phi,\eta,\xi)$ defines a left-invariant Sasakian structure on $H(1,n)$ and $(g,J)$ defines a left-invariant Vaisman structure on $H(1,n)\times\bR$.
Therefore, a Sasakian structure exists on every quotient of $H(1,n)$ by a lattice and a Vaisman structure exists on every quotient of $H(1,n)\times\bR$ by a lattice. As a particular choice of 
lattice in $H(1,n)$ one may take $\Lambda(1,n)$, consisting of the matrices with integer entries. Thus $\Lambda(1,n)\bs H(1,n)$ is a compact Sasakian nilmanifold. 
Taking the lattice $\Lambda(1,n)\times\bZ\subset H(1,n)\times\bR$, we obtain a Vaisman structure on the nilmanifold $(\Lambda(1,n)\times\bZ)\bs (H(1,n)\times\bR)$. Vaisman structures on compact solvmanifolds obtained as mapping tori of nilmanifolds quotient of $H(1,n)$ by a Sasakian automorphism were obtained in \cite{MP}.

In Section \ref{sec:3} we present examples of nilmanifolds and infra-nilmanifolds endowed with Vaisman structures.

\section{Main result}\label{sec:2}

In this section we prove the main result of this paper, namely:

\begin{theorem}\label{main}
 Let $V=\Gamma\bs G$ be a compact nilmanifold of dimension $2n+2$, $n\geq 1$, endowed with a Vaisman structure. Then $G$ is isomorphic to $H(1,n)\times\bR$.
\end{theorem}

\begin{proof}
By combining \cite[Structure Theorem]{OV} and \cite[Corollary 3.5]{OV2}, there exist a compact Sasakian manifold $S$ and a Sasakian automorphism $\varphi\colon S\to S$ such that $V$ is 
dif{}feomorphic to the mapping torus $S_\varphi$. Hence $V$ fibers over $S^1$ with fiber $S$ and the structure group of this fibration is generated by $\varphi$. A quick inspection of the long 
exact sequence of homotopy groups of $S\to V\to S^1$ shows that $S$ is an aspherical manifold and that $\Lambda\defeq\pi_1(S)$ sits in the short exact sequence
 \begin{equation}\label{fundamental_group}
  0\to \Lambda\to \Gamma\to\bZ\to 0.
 \end{equation}
Since $\Gamma$ is finitely generated, nilpotent and torsion-free, the same is true for $\Lambda$. By \cite[Theorem II.2.18]{Ragunathan}, we find a 
connected and simply connected nilpotent Lie group $H$ such that $\Lambda\subset H$ is a lattice and $P\defeq\Lambda\bs H$ is a compact nilmanifold. Notice that $P$ and $S$ 
are both aspherical manifolds with the same fundamental group; by standard results on aspherical manifolds (see for instance \cite{Lück}), they are homotopy equivalent. In particular, 
they have the same rational homotopy type. Since $\Lambda=\pi_1(P)=\pi_1(S)$ is nilpotent and both $P$ and $S$ are aspherical, they are nilpotent spaces. For nilpotent spaces, 
the rational homotopy type is codified in the minimal model, hence $P$ and $S$ have the same minimal model. 
 
Since $P$ is a nilmanifold, its minimal model is the Chevalley-Eilenberg complex $(\Lambda\,\fh^*,d)$, where $\fh$ is the Lie algebra of $H$. $S$ is not necessarily 
a nilmanifold, yet it is endowed with a Sasakian structure. In \cite[Theorem 1.1]{CMDNMY}, the authors prove that a compact nilmanifold of dimension $2n+1$, endowed with a Sasakian structure, 
is dif{}feomorphic to a quotient of the Heisenberg group $H(1,n)$ by a lattice. Their proof uses a certain property of the rational homotopy type of compact manifolds endowed with a Sasakian 
structure discovered by Tievsky (see \cite{Tievsky}). By carefully analyzing their arguments, one sees that they show more than what they claim: they prove that if a nilpotent manifold has a minimal model which is generated in degree 1, 
then such a minimal model is isomorphic to the minimal model of some compact quotient of the Heisenberg group $H(1,n)$ by a lattice, i.e.\ to the Chevalley-Eilenberg complex 
$(\Lambda\,\fh(1,n)^*,d)$, where $\fh(1,n)$ is the Lie algebra of $H(1,n)$. Now the minimal model of $S$, being isomorphic to that of $P$, is generated in degree 1, 
hence we conclude that $\fh$ is isomorphic to $\fh(1,n)$. Since $H$ and $H(1,n)$ are simply connected, they are isomorphic. From now on we replace $\fh$ by $\fh(1,n)$ and $H$ by $H(1,n)$. 
As a consequence, $\Lambda$ is contained in $H(1,n)$ as a lattice and $P=\Lambda\bs H(1,n)$.

By \cite[Corollary 3.4]{BMO}, there exists a finite cover $\pi\colon S\times S^1\to V$, say $m:1$, whose deck group $\bZ_m$ acts diagonally and by translations on the second factor; moreover, 
we have a commuting diagram of fiber bundles
\[
\xymatrix{
S\ar[r]\ar[d]_{\id} & S\times S^1\ar[d]_\pi\ar[r]^-{\text{pr}_2} & S^1\ar[d]^{\cdot m}\\
S\ar[r] & V\ar[r] & S^1
}
\]

where $\text{pr}_2$ is the projection on the second factor. Considering the induced maps on fundamental groups, we obtain a diagram of group homomorphisms
\[
\xymatrix{
1\ar[r] & \Lambda\ar[r]\ar[d]_{\id} & \Lambda\times \bZ\ar[d]_{\pi_*}\ar[r]^-{\text{pr}_2} & \bZ\ar[d]^{\cdot m}\ar[r] & 1\\
1\ar[r] & \Lambda\ar[r] & \Gamma\ar[r] & \bZ\ar[r] & 1
}
\]
By \cite[Theorem II.2.11]{Ragunathan}, this gives a third diagram at the level of the corresponding connected, simply connected nilpotent Lie groups:
\[
\xymatrix{
1\ar[r] & H(1,n)\ar[r]\ar[d]_{\id} & H(1,n)\times \bR\ar[d]_{\varpi}\ar[r]^-{\text{pr}_2} & \bR\ar[d]\ar[r] & 1\\
1\ar[r] & H(1,n)\ar[r] & G\ar[r] & \bR\ar[r] & 1
}
\]

where all the maps are continuous homomorphisms. Since all the nilpotent Lie groups appearing in this third diagram are simply connected, they admit no non-trivial covers, 
hence the homomorphism $\varpi$ is actually a group isomorphism.
\end{proof}

\begin{corollary}
 The minimal model of a Vaisman nilmanifold is isomorphic to the Chevalley-Eilenberg complex of the Lie algebra $\fh(1,n)\oplus\bR$.
\end{corollary}

\begin{remark}
A nilmanifold $N$, quotient of the Heisenberg group $H(1,n)$, is never formal, since it is not dif{}feomorphic to a torus. Moreover, it is easy to construct a triple Massey product on $N$ (refer to
\cite{OTralle} for Massey products). However, since $N$ admits a Sasakian structure, all Massey products of higher order on $N$ vanish by \cite[Theorem 4.4]{BFMT}. Similarly, a nilmanifold quotient of $H(1,n)\times\bR$ is not formal and has triple Massey products. It would be interesting to understand to which extent the existence of a Vaisman structure on a compact manifold influences the vanishing of higher order Massey products (see also \cite{OP}).
\end{remark}

\begin{remark}
A fortiori, the manifolds $P$ and $S$ in the proof of Theorem \ref{main} both admit Sasakian structures. Whether the homotopy equivalence $P\to S$ 
provides even a homeomorphism or not is a dif{}ficult question. In particular, it is not clear that $S$ is itself a nilmanifold. In the framework of aspherical manifolds, whether the homotopy equivalence lifts or not to a homeomorphism is the content 
of the Borel conjecture, which is known to be false in the smooth category - see \cite{Lück}. In particular, we do not claim that the homotopy equivalence $P\to S$ provides a 
Sasakian dif{}feomorphism $P\to S$
\end{remark}

\section{Two examples}\label{sec:3}

In the proof of Theorem \ref{main} we did not say anything about the order of the Sasakian automorphism. We give an example in which such order, albeit finite, can be arbitrarily large. 
We also construct an infra-nilmanifold, finitely covered by a nilmanifold quotient of $H(1,n)\times\bR$, endowed with a Vaisman structure. 

\subsection{Example 1}
We consider the lattice $\Lambda(1,n)\subset H(1,n)$ consisting of matrices with integer entries; we denote an element of $\Lambda(1,n)$ by 
$(a_1,b_1,\ldots,a_n,b_n,c)$. For $m\in\bN$, $m\geq 2$, let us take the sublattice
\[
\Delta(1,n;m)=\{(a_1,b_1,\ldots,a_n,b_n,c)\in\Lambda(1,n) \ | \ a_1\in m\bZ\}\subset\Lambda(1,n);
\]
then $\Delta(1,n;m)\bs H(1,n)$ is a nilmanifold and there is an $m:1$ covering $\Delta(1,n;m)\bs H(1,n)\to N(1,n)\defeq\Lambda(1,n)\bs H(1,n)$ with deck group 
$\Delta(1,n;m)\bs\Lambda(1,n)\cong\bZ_m$. Being quotients of $H(1,n)$, both nilmanifolds have Sasakian structures, coming from the left-invariant Sasakian structure of $H(1,n)$. 
We denote by $\varphi$ the generator of $\bZ_m$; the action of $\varphi$ on $\Delta(1,n;m)\bs H(1,n)$ is covered by the left-translation
\[
 (x_1,y_1,\ldots,x_n,y_n,z)\mapsto(x_1+1,y_1,\ldots,x_n,y_n,z)
\]
on $H(1,n)$; since the Sasakian structure is left-invariant, this action is by Sasakian automorphisms. 
On $\Delta(1,n;m)\bs H(1,n)\times\bR$, we consider the $\bZ$-action generated by
\[
1\cdot(s,t)=(\varphi(s),t+1) \quad \text{for} \ (s,t)\in \Delta(1,n;m)\bs H(1,n)\times\bR.
\]
The mapping torus $(\Delta(1,n;m)\bs H(1,n))_\varphi$ has a Vaisman structure. By construction, $(\Delta(1,n;m)\bs H(1,n))_\varphi$ is dif{}feomorphic to the nilmanifold $N(1,n)\times S^1$.

\subsection{Example 2}
\emph{Infra-nilmanifolds} are quotients of connected and simply connected nilpotent Lie groups which are more general than nilmanifolds. They can be constructed as follows.
One starts with a connected, simply connected nilpotent Lie group $G$ and considers the group $\Aut(G)$ of automorphisms of $G$. The af{}fine group $\Aff(G)\defeq G\rtimes\Aut(G)$ acts on $G$ as 
follows:
\begin{equation}\label{eq:2}
 (g,\varphi)\cdot h=g\varphi(h), \quad g,h\in G, \ \varphi\in\Aut(G).
\end{equation}

One takes then a compact subgroup $K\subset\Aut(G)$ and any discrete, torsion-free subgroup $\Xi\subset G\rtimes K$, such that $\Xi\bs G$ is compact; here $\Xi$ acts on 
$G$ according to \eqref{eq:2}. The action of $\Xi$ on $G$ is free and properly discontinuous, hence $\Xi\bs G$ is a manifold. By definition, $\Xi\bs G$ is an infra-nilmanifold
(see \cite{Dekimpe}). By \cite[Theorem 1]{Auslander}, $\Xi\cap G$ is a uniform lattice in $G$ and $(\Xi\cap G)\bs \Xi$ is a finite group, hence the fundamental group of an 
infra-nilmanifold is virtually nilpotent. The finite group $(\Xi\cap G)\bs \Xi$ is called the \emph{holonomy group} of $\Xi$. If it is trivial, i.e.\ $\Xi\subset G$, then $\Xi\bs G$ is a 
nilmanifold. In general, an infra-nilmanifold $\Xi\bs G$ is finitely covered by a nilmanifold $(\Xi\cap G)\bs G$.

Again we consider the Heisenberg group $H(1,n)$, $n\geq 2$, and the lattice $\Lambda(1,n)$. The map $\phi\colon \Lambda(1,n)\to\Lambda(1,n)$, 
\[
(a_1,b_1,a_2,b_2,\ldots,a_n,b_n,c)\mapsto (a_2,b_2,a_1,b_1,\ldots,a_n,b_n,c) 
\]
is a group automorphism of order 2, which extends to a group automorphism $H(1,n)\to H(1,n)$ denoted again $\phi$. We consider the lattice $\Gamma\defeq\Lambda(1,n)\times\bZ\subset G\defeq H(1,n)\times\bR$, the af{}fine group $\Aff(G)=G\rtimes\Aut(G)$ and the subgroup of $\Aut(G)$ generated by the automorphism $\sigma\colon G\to G$,
\[
 \sigma(h,t)=(\phi(h),t), \quad (h,t)\in G.
\]
Then $\langle\sigma\rangle\subset \Aut(G)$ is isomorphic to $\bZ_2$, hence compact. Finally, we take the lattice $\Xi=\Gamma\rtimes\langle\sigma\rangle\subset\Aff(G)$ and 
consider the infra-nilmanifold $V=\Xi\bs G$. Clearly $\Xi\cap G=\Gamma$ and $\Gamma\bs\Xi\cong\bZ_2$, hence the holonomy of $\Xi$ is non-trivial and $V$ is covered $2:1$ by the nilmanifold $\Gamma\bs G$.

Finally, we show that $V$ can be seen as the mapping torus of a Sasakian automorphism of order 2 of $N(1,n)=\Lambda(1,n)\bs H(1,n)$, hence it carries a Vaisman structure. First of all,
by construction the automorphism $\phi\colon H(1,n)\to H(1,n)$ preserves $\Lambda(1,n)$, hence descends to a dif{}feomorphism $\varphi\colon N(1,n)\to N(1,n)$. One checks easily that $\varphi$ preserves the Sasakian structure. The mapping torus $N(1,n)_\varphi$, 
defined as the quotient of $N(1,n)\times\bR$ by the $\bZ$-action given by \eqref{mapping:torus}, carries a Vaisman structure. Clearly $N(1,n)_\varphi$ is dif{}feomorphic to the infra-nilmanifold $V$.

\section*{Acknowledgements}
I am indebted to S\"onke Rollenske for a decisive conversation and for bringing infra-nilmanifolds to my attention. 
I would also like to thank Juan Carlos Marrero, Vicente Mu\~noz and John Oprea for their useful comments on an earlier version of this paper.


\end{document}